\documentclass[a4paper,11pt,fleqn]{article}

\usepackage{amsmath,amssymb,amsthm,enumerate}%,cite,backref
\usepackage[mathscr]{eucal}

\usepackage{hyperref}
\hypersetup{colorlinks, linkcolor=blue, citecolor=blue, urlcolor=blue}

\newcommand{\p}{\partial}

\newtheorem{theorem}{Theorem}%[section]
\newtheorem{lemma}[theorem]{Lemma}

\newtheorem*{proposition*}{Proposition}
{\theoremstyle{definition}

\newtheorem*{notation*}{Notation}
}

\newcommand{\todo}[1][\null]{\ensuremath{\clubsuit}}

\newcommand{\noprint}[1]{}

\begin{document}

\par\noindent {\LARGE\bf
Variational symmetries and conservation laws\\ of the wave equation in one space dimension\par}

{\vspace{3mm}\par\noindent\large
Roman O. Popovych~$^\dag$ and Alexei F. Cheviakov~$^\ddag$
\par}

{\vspace{2mm}\par\noindent\it
$^{\dag}$Fakult\"at f\"ur Mathematik, Universit\"at Wien, Oskar-Morgenstern-Platz 1, 1090 Wien, Austria\\
$\phantom{^\dag}$Institute of Mathematics of NAS of Ukraine, 3 Tereshchenkivska Str., 01024 Kyiv, Ukraine
\par}

{\vspace{1mm}\noindent\it
$^\ddag$Department of Mathematics and Statistics, University  of Saskatchewan, Room 227 McLean Hall, 106 Wiggins Road, Saskatoon, SK, S7N 5E6 CANADA
\par}

\vspace{1mm}\par\noindent 
E-mail:  $^\dag$rop@imath.kiev.ua, $^\ddag$shevyakov@math.usask.ca

\vspace{4mm}\par\noindent\hspace*{5mm}\parbox{150mm}{\small
The direct method based on the definition of conserved currents of a system of differential equations
is applied to compute the space of conservation laws of the (1+1)-dimensional wave equation in the light-cone coordinates.
Then Noether's theorem yields the space of variational symmetries of the corresponding functional.
The results are also presented for the standard space-time form of the wave equation.}

\vspace{3mm}

\noindent 
{\footnotesize Keywords: wave equation, variational symmetry, conservation law, Noether's theorem, direct method%, conservation-law characteristic
\par\vspace{3mm}}

\noprint{
MSC: 35L05, 35B06, 37K05
35-XX   Partial differential equations
 35Bxx  Qualitative properties of solutions
  35B06   Symmetries, invariants, etc.
 35Lxx		Hyperbolic equations and systems [See also 58J45]
  35L05  	Wave equation
70-XX   Mechanics of particles and systems {For relativistic mechanics, see 83A05 and 83C10; for statistical mechanics, see 82-XX}
 70Sxx  Classical field theories [See also 37Kxx, 37Lxx, 78-XX, 81Txx, 83-XX]
  70S10   Symmetries and conservation laws
37-XX   Dynamical systems and ergodic theory [See also 26A18, 28Dxx, 34Cxx, 34Dxx, 35Bxx, 46Lxx, 58Jxx, 70-XX]
 37Kxx  Infinite-dimensional Hamiltonian systems [See also 35Axx, 35Qxx]
  37K05   Hamiltonian structures, symmetries, variational principles, conservation laws

}

\section{Introduction}\label{sec:1}

The linear (1+1)-dimensional constant-coefficient wave equation in space-time variables
\begin{equation}\label{eq:w:tx}
\mathcal U\colon\quad u_{tt}-u_{xx}=0
\end{equation}
introduced by d'Alembert in 1747 \cite{d1747recherches} is a fundamental hyperbolic partial differential equation (PDE)
describing small oscillations in a wide variety of physical settings.
(In \eqref{eq:w:tx} and below, subscripts denote partial derivatives.)
The general solution of~\eqref{eq:w:tx} possesses a simple closed-form representation.
For many initial-boundary value problems for~\eqref{eq:w:tx},
the corresponding solutions can be constructed by standard methods
such as the method of characteristics, separation of variables, integral transforms, or Green's functions.
In the light-cone (characteristic) variables
\begin{equation}\label{eq:chvar}
\xi=x+t, \quad \eta=x-t,\quad u(t,x)=w(\xi,\eta),
\end{equation}
the PDE \eqref{eq:w:tx} assumes the second canonical form
\begin{equation}\label{eq:w:dal}
\mathcal W\colon\quad w_{\xi\eta}=0.
\end{equation}
The wave equation arises from a variational principle, with the Lagrangian densities respectively given by
$\mathrm L_{\mathcal U}=-\frac12(u_t^2-u_x^2)$ and $\mathrm L_{\mathcal W}=-\frac12w_\xi w_\eta$.
Generalized symmetries (or, equivalently, cosymmetries) of the wave equation were completely described in \cite[Section~18.4]{Ibragimov1985}
using the representation~\eqref{eq:w:dal} of this equation in the light-cone variables.
At the same time, there are no complete results for variational symmetries% 
\footnote{%
Given a functional with a Lagrangian~$\mathrm L$ being a differential function of~$u$, 
it is convenient for us to call variational symmetries \cite[Definition~5.51]{Olver1993} of this functional 
as variational symmetries of the corresponding system of Euler--Lagrange equations $\mathsf E_uL=0$, 
where $\mathsf E_u$ denotes the Euler operator with respect to~$u$ \cite[Definition~4.3]{Olver1993}.
}
and local conservation laws of the wave equation
in the literature although a number of its conservation laws including physically important ones are commonly known.
The goal of the present paper is to fill this gap in the literature.

Although both the light-cone form \eqref{eq:w:dal} and the space-time form \eqref{eq:w:tx} of the wave equation 
are normal systems of differential equations in the sense of \cite[Definition 2.78]{Olver1993}, 
yet an important difference between them is 
that \eqref{eq:w:tx} can straightforward be represented in the \emph{(extended) Kovalevskaya form} \cite{MAlonso1979,Tsujishita1982} 
(see also footnote in~\cite{Popovych&Bihlo2020})
with respect to each of the independent variables $t$ and~$x$, 
which is not the case for the PDE \eqref{eq:w:dal} as it stands.
As shown in~\cite{MAlonso1979}, if a normal system of differential equations is written in the extended Kovalevskaya form with respect to an independent variable, 
then every equivalence class of conservation-law characteristics contains a characteristic that does not depend on the corresponding principal derivatives.%
\footnote{%
See also \cite{Anco&Bluman2002b,Bocharov&Co1999,Bluman&Cheviakov&Anco2010,Ibragimov1985,Olver1993,Popovych&Bihlo2020,Tsujishita1982}
for basic theoretical results on conservation laws of systems of differential equations, 
methods of their computations and further references. 
} 
There are no similar results for non-Kovalevskaya rankings of derivatives. 
Hence the quotient space of conservation-law characteristics of \eqref{eq:w:tx}
is naturally isomorphic to the space of differential functions $\mu=\mu[u]$ 
that satisfy the condition 
\begin{equation}\label{eq:det:CL:xt}
\mathsf E_u(\mu(u_{tt}-u_{xx}))\equiv 0
\end{equation}
and are assumed, e.g., to be independent of all derivatives $u_{ij}$, $i\geqslant2$, $j\geqslant0$. 
Here and in what follows $u_{ij}:=\p^{i+j} u/\p^it\p^jx$ and $w_{ij}:=\p^{i+j} w/\p^i\xi\p^j\eta$, 
$i,j\in\mathbb N_0:=\mathbb N\cup\{0\}$.
The determining equations obtained by splitting the condition~\eqref{eq:det:CL:xt} prove to be, however, rather hard to deal with. 
The same happens with the computation of conservation-law characteristics for the light-cone form~\eqref{eq:w:dal} of the wave equation. 
Here things are further complicated by the fact that \eqref{eq:w:dal} does not have an extended Kovalevskaya form, 
and hence in the course of excluding the principal derivative $w_{kl}$, $k,l\geqslant1$, for avoiding singular multipliers,  
one might lose conservation laws.

\looseness=-1
This is why in Section \ref{sec:2} we use the brute-force approach, 
working directly with conserved currents
to compute all local conservation laws of the wave equation~\eqref{eq:w:dal} 
in the light-cone variables. %$(\xi,\eta)$. 
Then using the characteristic form of conservation laws
we find the conservation-law characteristics (=~variational symmetries) of~\eqref{eq:w:dal}, 
which shows that they constitute a proper subspace of the space of all cosymmetries (=~generalized symmetries) of~\eqref{eq:w:dal}. 
The bijection between equivalence classes of conservation-law characteristics and those of conserved currents \cite[Theorem~4.26]{Olver1993}  
leads to the description of the space of trivial conserved currents of~\eqref{eq:w:dal} 
and, thus, to the complete characterization of the space of local conservation laws of~\eqref{eq:w:dal} 
as a quotient space.  

In Section \ref{sec:3}, local conservation laws of the wave equation~\eqref{eq:w:tx} in space-time variables 
are explicitly derived using the results of Section \ref{sec:2} and the transformation~\eqref{eq:chvar}. 
In particular, it is shown that the conservation-law characteristics of~\eqref{eq:w:tx} have a rather complex form, 
which is unlikely to be obtained in the closed form using the condition~\eqref{eq:det:CL:xt}. 
Some examples are considered.

The paper is concluded with a discussion. 
An important question we are concerned with, for the current example and in general, 
is the equivalence of conservation-law characteristics on the solution space of the given system of differential equations, 
and the possibility to exclude higher-rank derivatives the system is solved for 
from the characteristic dependence by substitutions on the solution space. 
This question, in view of the possible loss of conservation laws after such exclusions, 
is discussed for different forms of the wave equation. % and the related Klein--Gordon equation.

\section{Computation in light-cone coordinates}\label{sec:2}

The computation of generalized symmetries, conservation laws and variational symmetries of the (1+1)-dimensional wave equation
is convenient to carry out in the light-cone variables $(\xi,\eta)$.

For the computational purposes, it is opportune to use the following particular notation.
The notation $F=F[w|\eta]$ and $G=G[w|\xi]$ for differential functions~$F$ and~$G$ respectively mean
that $F$ is a smooth function of~$\eta$ and a finite number of $w_{0l}$, $l\in\mathbb N$,
and $G$ is a smooth function of~$\xi$ and a finite number of $w_{k0}$, $k\in\mathbb N$,
\[
F=F(\eta,w_{01},\dots,w_{0r_1}), \quad
G=G(\xi,w_{10},\dots,w_{r_20}), \quad r_1,r_2\in\mathbb N_0.
\]
$\mathscr D_\xi:=\p_{\xi}+w_{k+1,0}\p_{w_{k0}}$ and $\mathscr D_\eta:=\p_{\eta}+w_{0,l+1}\p_{w_{0l}}$
are the evaluations of the total derivative operators~$\mathrm D_\xi$ and~$\mathrm D_\eta$ 
with respect to~$\xi$ and~$\eta$ on solutions of~$\mathcal W$, respectively.
Here and in what follows
the indices $k$ and~$l$ run through~$\mathbb N$% unless otherwise stated
, and we assume summation for repeated indices.

Due to the existence of a Lagrangian formulation,
the linearization operator for the equation~$\mathcal W$ is self-adjoint,
and hence the spaces of symmetries and cosymmetries of~$\mathcal W$ coincide.
As was mentioned in the introduction, the algebra of generalized symmetries of~$\mathcal W$ is well known \cite[Section~18.4]{Ibragimov1985}.
Any generalized symmetry of~$\mathcal W$ is equivalent to a generalized symmetry in the evolution form $\zeta\p_w$
with $\zeta=Cw+F[w|\eta]+G[w|\xi]$,
where $C$ is an arbitrary constant, and $F$ and~$G$ are arbitrary functions of the above kind. 
Such structure of generalized symmetries is not common for linear systems of differential equations~\cite{Shapovalov&Shirokov1992}.

\looseness=-1
Not all generalized symmetries of~$\mathcal W$ are variational symmetries of the associated functional with Lagrangian~$\mathrm L_{\mathcal W}$ (in fact, relatively few are).
We now describe the entire algebra of variational symmetries (or, equivalently, the space of conservation-law characteristics) of the equation~$\mathcal W$.
It is not easy to single out this algebra (or this space) from the algebra of generalized symmetries of~$\mathcal W$
(or from the space of cosymmetries of~$\mathcal W$).
This is why we first compute a space of conserved currents representing conservation laws of~$\mathcal W$
and then obtain the associated space of conservation-law characteristics of~$\mathcal W$,
which is converse to the common use of Noether's theorem.

\begin{lemma}\label{lem:wDal:GenFormOfCCs}
Any conserved current of the equation~\eqref{eq:w:dal} is equivalent to
a tuple $(F[w|\eta],G[w|\xi])$.
\end{lemma}

\begin{proof}
Let $(F,G)$ be a conserved current of the equation~$\mathcal W$.
Without loss of generality, up to the conserved-current equivalence related to vanishing on the solution set of~$\mathcal W$,
all the mixed derivatives can be excluded from $(F,G)$.
In other words, we can assume that the tuple $(F,G)$ depends at most on $\xi$, $\eta$, $w$
and a finite number of non-mixed derivatives $w_{k0}$ and $w_{0l}$ and $k,l\in\mathbb N$.
The condition $(\mathrm D_\xi F+\mathrm D_\eta G)\big|_{\mathcal W}=0$ for conserved currents of~$\mathcal W$
can then be rewritten in the form $\mathscr D_\xi F+\mathscr D_\eta G=0$.
We fix an arbitrary point $\mathbf j^0=(\xi^0,\eta^0,w_{kl}^0,k,l\in\mathbb N_0)$ in the domain of $(F,G)$.
(Only a finite number of components of~$\mathbf j^0$ is relevant for the proof.)

First we prove that up to the general conserved-current equivalence,
the $\xi$-component~$F$ does not depend on $\xi$ and $\xi$-derivatives of $w$, including $w$ as the zeroth-order derivative.
Assume that it does, and the highest order of such derivatives is nonnegative, 
\mbox{$q:=\max\{k\in\mathbb N_0\mid F_{w_{k0}}\ne 0\}\geqslant 0$}.
We differentiate the equation $\mathscr D_\xi F+\mathscr D_\eta G=0$ with respect to $w_{q+1,0}$
and use the fact that $\p_{w_{q+1,0}}$ and $\mathscr D_\eta$ commute, which gives $\p_{w_{q0}}F+\mathscr D_\eta\p_{w_{q+1,0}} G=0$.
Then the integration of the obtained equality with respect to the jet variable~$w_{q0}$ in a neighborhood of~$\mathbf j^0$
results in the following representation for~$F$:
\[
F=\widetilde F-\mathscr D_\eta H,
\quad\mbox{where}\quad
\widetilde F:=F\big|_{w_{q0}^{}=w_{q0}^0}, \quad
H:=\int_{w_{q0}^0}^{w_{q0}}(\p_{w_{q+1,0}}G)\big|_{w_{q0}=\tau}\,{\rm d}\tau.
\]
Let $\tilde q$ be the highest order of $\xi$-derivatives of $w$ involved in~$\widetilde F$,
$\tilde q:=\max\{k\in\mathbb N_0\mid\widetilde F_{w_{k0}}\ne 0\}$ if this set is nonempty and $\tilde q:=-\infty$ otherwise.
The differential function~$\widetilde F$ does not depend on~$w_{q0}$, so we have $\tilde q<q$.
The tuple $(\widetilde F,\widetilde G)$ with $\widetilde G:=G-\mathscr D_\xi H$
is a conserved current of~$\mathcal W$ that is equivalent to $(F,G)$
since it is obtained from $(F,G)$ by adding the trivial conserved current $(\mathscr D_\eta H,-\mathscr D_\xi H)$.
Substituting zeros into $\widetilde G$ for the involved mixed derivatives,
we construct the conserved current $(\widetilde F,\widehat G)$,
which is equivalent to $(\widetilde F,\widetilde G)$
in view of the vanishing difference of these conserved currents on the solution set of~$\mathcal W$
and thus equivalent to $(F,G)$.
We replace $(F,G)$ by $(\widetilde F,\widehat G)$, omitting the accent signs.
The repetition of the above procedure leads to an equivalent conserved current,
still denoted by $(F,G)$,
with $\xi$-component~$F$ involving no $\xi$-derivatives $w_{k0}$, $k\in\mathbb N_0$.
For this conserved current, the condition $(\mathrm D_\xi F+\mathrm D_\eta G)\big|_{\mathcal W}=0$
degenerates to $F_{\xi}+\mathscr D_\eta G=0$.
Now we employ a modification of the above procedure for~$\xi$.
The integration of the latter equality with respect to~$\xi$ in a neighborhood of~$\mathbf j^0$ gives
\[
F=\widetilde F-\mathscr D_\eta H,
\quad\mbox{where}\quad
\widetilde F:=F\big|_{\xi=\xi^0}, \quad
H:=\int_{\xi^0}^{\xi}G\big|_{\xi=\tau}\,{\rm d}\tau.
\]
The differential function~$\widetilde F$ does not depend on~$\xi$.
The tuple $(\widetilde F,\widetilde G)$ with $\widetilde G:=G-\mathscr D_\xi H$
is a conserved current of~$\mathcal W$ that is equivalent to $(F,G)$
since it is obtained from $(F,G)$ by adding the trivial conserved current $(\mathscr D_\eta H,-\mathscr D_\xi H)$.
Substituting zeros into $\widetilde G$ for the involved mixed derivatives,
we construct the conserved current $(\widetilde F,\widehat G)$,
which is equivalent to $(\widetilde F,\widetilde G)$
in view of vanishing the difference of these conserved currents on the solution set of~$\mathcal W$
and thus equivalent to $(F,G)$.
We replace $(F,G)$ by $(\widetilde F,\widehat G)$, again omitting the accent signs.

Since the $\xi$-component of the new conserved current~$(F,G)$
depends only on $\eta$ and a finite number of~$w_{0l}$, $l\in\mathbb N$,
i.e., $F=F[w|\eta]$,
then the $\eta$-component~$G$ satisfies the equation $\mathscr D_\eta G=0$,
which recursively splits into the system $G_\eta=0$, $G_{w_{0l}}=0$, $l\in\mathbb N_0$.

Varying the point $\mathbf j^0$ within the domain of $(F,G)$ completes lemma's proof.
\end{proof}

\begin{lemma}\label{lem:wDal:chars}
The quotient space of variational symmetries (or, equivalently, of conservation-law characteristics)
of the equation~\eqref{eq:w:dal} is naturally isomorphic to the subspace of its variational symmetries of the form
\begin{gather}\label{eq:w:dal:CL:char}
\lambda=(-\mathscr D_\eta)^{l-1}\p_{w_{0l}}F[w|\eta] + (-\mathscr D_\xi)^{k-1}\p_{w_{k0}}G[w|\xi].
\end{gather}
\end{lemma}

\begin{proof}
We fix an arbitrary conservation law of the equation~$\mathcal W$.
In view of Lemma~\ref{lem:wDal:GenFormOfCCs}, it contains a conserved current of the form $(F[w|\eta],G[w|\xi])$.
We expand the total divergence of $(F,G)$ and iteratively integrate by parts in
(or, equivalently, apply the Lagrange identity to) each obtained summand:
\begin{eqnarray*}
\mathrm D_\xi F + \mathrm D_\eta G &=& w_{1l} \p_{w_{0l}}F + w_{k1} \p_{w_{k0}}G  \\
&=& \left((-\mathscr D_\eta)^{l-1}\p_{w_{0l}}F + (-\mathscr D_\xi)^{k-1}\p_{w_{k0}}G\right) w_{11} + \mathrm D_\xi F^0+\mathrm D_\eta G^0,
\end{eqnarray*}
where $(F^0,G^0)$ is a trivial conserved current, which vanishes on solutions of~$\mathcal W$.
We derive the characteristic representation of the fixed conservation law
in terms of the conserved current $(F-F^0,G-G^0)$, which is equivalent to $(F,G)$,
\[
\mathrm D_\xi(F-F^0)+\mathrm D_\eta (G-G^0)=\lambda w_{11}
\]
with $\lambda$ given by \eqref{eq:w:dal:CL:char}.
Therefore, $\lambda$ is a characteristic of the fixed conservation law
as well as a variational symmetry of~$\mathcal W$.

It is easy to check that any tuple $(F[w|\eta],G[w|\xi])$ is a conserved current of~$\mathcal W$.
Hence any differential function of the form~\eqref{eq:w:dal:CL:char}
is a conservation-law characteristic of~$\mathcal W$.
Such a differential function vanishes on solutions of~$\mathcal W$
if and only if it identically vanishes.
\end{proof}

In particular, any function $\lambda=\lambda(\xi,\eta)$ satisfying the equation~\eqref{eq:w:dal},
$\lambda_{\xi\eta}=0$, i.e., $\lambda=F(\eta)+G(\xi)$,
is a $w$-independent (order $-\infty$) variational symmetry of this equation.

Let $V$ denote the subspace of conserved currents $(F,G)$ of equation~\eqref{eq:w:dal},
where $F$ is an arbitrary smooth function of~$\eta$ and a finite but unspecified number of $w_{0l}$, $l\in\mathbb N$,
and $G$ is an arbitrary smooth function of~$\xi$ and a finite but unspecified number of $w_{k0}$, $k\in\mathbb N$,
$V:=\{(F[w|\eta],G[w|\xi])\}$.
As mentioned in the proof of Lemma~\ref{lem:wDal:chars},
all elements of~$V$ are indeed conserved currents of~\eqref{eq:w:dal}.
Let $V_0$ denote the subspace of trivial conserved currents belonging to~$V$.

\begin{lemma}\label{lem:wDalTrivCCs}
The subspace~$V_0$ consists of the tuples that can be locally represented in the form $(F,G)$,
where $F=\mathscr D_\eta\tilde F+Cw_{01}$ and $G=\mathscr D_\xi\tilde G-Cw_{10}$
for some $\tilde F=\tilde F[w|\eta]$, some $G=\tilde G[w|\xi]$ and some constant~$C$.
\end{lemma}

\begin{proof}
It is obvious that $\mathcal W$ is normal, totally nondegenerate system of differential equations.
Then in view of \cite[Theorem 4.26]{Olver1993} and Lemma~\ref{lem:wDal:chars},
a conserved current $(F[w|\eta],G[w|\xi])$ of~$\mathcal W$ is trivial
if and only if the differential function defined by~\eqref{eq:w:dal:CL:char} identically vanishes.
Since the functions~$F$ and~$G$ depend on different arguments,
the latter condition is equivalent to that the expressions
$R^1:=(-\mathscr D_\eta)^{l-1}\p_{w_{0l}}F$ and $R^2:=(-\mathscr D_\xi)^{k-1}\p_{w_{k0}}G$
are constants with $R^2=-R^1$.
Here the operator $(-\mathscr D_\eta)^{l-1}\p_{w_{0l}}$ (resp.\ $(-\mathscr D_\xi)^{k-1}\p_{w_{k0}}$)
can be interpreted as the Euler operator in the dependent variable~$w_{01}$ (resp.\ $w_{10}$),
where the only $\eta$ (resp.\ $\xi$) is assumed to be the independent variable,
and $\xi$ (resp.\ $\eta$) plays the role of a parameter.
In view of \cite[Theorem 4.7]{Olver1993}, the defining equalities for~$R^1$ and~$R^2$ 
considered as linear inhomogeneous equations with respect to~$F$ and~$G$
are equivalent to the (local) representations for~$F$ and~$G$ from lemma's statement.
\end{proof}

As a consequence of Lemmas~\ref{lem:wDal:GenFormOfCCs} and~\ref{lem:wDalTrivCCs},
we derive the following theorem.

\begin{theorem}\label{thm:wDal:chars}
The space of conservation laws of the equation~\eqref{eq:w:dal}
is naturally isomorphic to the quotient of the space~$V$ by the subspace~$V_0$.
\end{theorem}

\section{Reformulation in standard space-time coordinates}\label{sec:3}

Assuming the wave equation in the space-time form~$\mathcal U$ solved with respect to $u_{tt}$
and substituting $\p_\xi=\frac12(\p_x+\p_t)$, $\p_\eta=\frac12(\p_x-\p_t)$ and
$u_{2k',l}=u_{0,l+2k'}$, $u_{2k'+1,l}=u_{1,l+2k'}$ in view of~$\mathcal U$,
we obtain that $w_{k0}=\frac12(u_{0k}+u_{1,k-1})$ and $w_{0k}=\frac12(u_{0k}-u_{1,k-1})$
for solutions of~$\mathcal W$ and~$\mathcal U$ related by the change of variables~\eqref{eq:chvar}.
Therefore, translating the results for the equation~$\mathcal W$ from the previous section
to ones for the equation~$\mathcal U$ via pulling back by the transformation~\eqref{eq:chvar},
we derive the following assertion.

\begin{theorem}\label{thm:wave:CLsAndVarSyms}
The space of conservation laws of the equation~\eqref{eq:w:tx}
is naturally isomorphic to the quotient of the space~$\check V$ by the subspace~$\check V_0$,
where $\check V$ denotes the subspace of conserved currents of the form
$(\check F-\check G,\check F+\check G)$ for the equation~\eqref{eq:w:tx},
and $\check V_0$ denotes the subspace of trivial conserved currents belonging to~$\check V$,
\[\check V:=\{(\check F-\check G,\check F+\check G)\}\supset
\check V_0:=\{(\check{\mathscr D}_\eta\check F-\check{\mathscr D}_\xi\check G+Cu_{01},\check{\mathscr D}_\eta\check F+\check{\mathscr D}_\xi\check G-Cu_{10})\}.\]
Here and in what follows
$\check F$ is an arbitrary smooth function of~$\eta:=x-t$
and an arbitrary finite number of $\check w_{0l}:=u_{0l}-u_{1,l-1}$, $l\in\mathbb N$,
and $\check G$ is an arbitrary smooth function of~$\xi:=x+t$
and an arbitrary finite number of $\check w_{k0}:=u_{0k}+u_{1,k-1}$, $k\in\mathbb N$,
\begin{gather*}%\label{FGu}
\check F=\check F[\check w|\eta]=\check F(\eta,\check w_{01},\dots,\check w_{0r_1}) = \check F(x-t,u_{01}-u_{10}, \ldots,u_{0 r_1}-u_{1,r_1-1}),\\
\check G=\check G[\check w|\xi ]=\check G(\xi ,\check w_{10},\dots,\check w_{r_20}) = \check G(x+t,u_{01}+u_{10}, \ldots,u_{0 r_2}+u_{1,r_2-1})
\end{gather*}
for some $ r_1,r_2\in\mathbb N_0$,
$\check{\mathscr D}_\xi:=\p_\xi+\check w_{k+1,0}\p_{\check w_{k0}}$ and $\check{\mathscr D}_\eta:=\p_\eta+\check w_{0,l+1}\p_{\check w_{0l}}$, 
and $C$ is an arbitrary constant. 
The quotient space of variational symmetries (or, equivalently, of conservation-law characteristics)
of the equation~\eqref{eq:w:tx} is naturally isomorphic to the subspace of its variational symmetries of the form
\[
\mu=-(-\check{\mathscr D}_\eta)^{l-1}\p_{\check w_{0l}}\check F[\check w|\eta]-(-\check{\mathscr D}_\xi)^{k-1}\p_{\check w_{k0}}\check G[\check w|\xi].
\]
\end{theorem}

To avoid multiple twos and halves in the above formulas, we multiply~$\check w_{kl}$, 
in comparison with~$w_{kl}$, by 2. 
The additional minuses in~$\mu$ is needed for associating~$\mu$ 
with the conserved current $(\check F-\check G,\check F+\check G)$.

%see file wave_GenFormula_Lambda_Dal_and_XT.mw
In particular, the solutions $\mu=\mu(t,x)$ of the equation~\eqref{eq:w:tx}, $\mu_{tt}-\mu_{xx}=0$,
exhaust the subspace of $u$-independent (order $-\infty$) variational symmetries of this equation.
Basic examples of physical conservation laws describe
the conservation of momentum in the $u$-direction,
the motion of the center of mass of the wave,
the conservation of angular momentum with respect to $x=0$
and the conservation of energy,
which respectively have the characteristics $\mu=1$, $t$, $x$, $u_t$
and the conserved currents $(u_t,-u_x)$, $(tu_t-u,-tu_x)$, $(xu_t,-xu_x+u)$, $(\frac12u_t^{\,2}+\frac12u_x^{\,2},-u_t u_x)$.
These conservation laws arise for the pairs $(\check F,\check G)=-\frac12(\check w_{01},\check w_{10})$,
$\frac12(\eta\check w_{01},-\xi\check w_{10})$, $-\frac12(\eta\check w_{01},\xi\check w_{10})$,
and $\frac14(\check w_{01}^{\,\,\,2},-\check w_{10}^{\,\,\,2})$, respectively.
Other simple conservation-law characteristics of~\eqref{eq:w:tx} are, e.g., $\mu=tx$, $u_x$, $xu_x+tu_t$ and $xu_t+tu_x$. %(see, e.g., \cite{DCC})
A~more unusual choice, e.g., of the pair $(\check F,\check G)=(e^{\check w_{02}},0)$ yields the local conservation law of~\eqref{eq:w:tx}
with conserved current $(e^{u_{02}-u_{11}},e^{u_{02}-u_{11}})$ %$(e^{u_{02}-u_{11}},({u_{20}-u_{02}+1}) e^{u_{02}-u_{11}})$
and with characteristic $\mu=(u_{03}-u_{12})e^{u_{02}-u_{11}}$.
This corresponds to the local conservation law with conserved current
$(e^{2w_{02}},0)$ % $\sim(e^{2w_{02}},-2w_{11}e^{2w_{02}})$
and with characteristic $\lambda=-4w_{03}e^{2w_{02}}$
of the wave equation \eqref{eq:w:dal} in the light-cone variables.

\section{Discussion}\label{sec:Disc}

Let us discuss the possibility of describing the space of local conservation laws of the equation~\eqref{eq:w:dal}
using the criterion for its conservation-law characteristics $\lambda[w]$, 
\begin{equation}\label{eq:CriterionForLambda}
\mathsf E_w(\lambda[w]w_{\xi\eta})\equiv 0.
\end{equation}
\looseness=-1
Since the PDE \eqref{eq:w:dal} is not in a extended Kovalevskaya form, 
here one is generally not allowed to exclude the dependence of~$\lambda$ 
on the only possible leading derivative $w_{\xi\eta}$ and its further derivatives, 
and so it is not clear how to avoid ``singular multipliers''. 
However, Lemma~\ref{lem:wDal:chars} implies 
that in this particular case, 
each equivalence class of conservation-law characteristics of~\eqref{eq:w:dal} 
contains a characteristic of the form~\eqref{eq:w:dal:CL:char}, 
which is independent of the mixed derivatives $w_{kl}$, $k,l\in\mathbb N$, 
which exhaust the principal derivatives for~\eqref{eq:w:dal}.
Even after neglecting the problem of excluding principal derivatives, 
it is not straightforward to find the explicit general form~\eqref{eq:w:dal:CL:char} 
for conservation-law characteristics of~\eqref{eq:w:dal} by solving~\eqref{eq:CriterionForLambda} directly. 
Indeed, even assuming $\lambda[w]=H(\xi,w_{10},\ldots,w_{r0})$ with $r\geqslant1$ 
(which should yield the second part of the formula \eqref{eq:w:dal:CL:char}: $H=(-\mathscr D_\xi)^{k-1}\p_{w_{k0}}G[w|\xi]$), 
the condition~\eqref{eq:CriterionForLambda} splits with respect to derivatives $w_{k1}$, $k\in\mathbb N$, into the system
\[
\p_{w_{i-1,0}}H+\mathscr D_\xi\p_{w_{i0}}H+\sum_{j=i-1}^{r}(-1)^j \binom{j}{i-1}\mathscr D_\xi^{j-i+1}\p_{w_{j0}}H=0,\quad i=1,\ldots, r+1,
\]
which remains quite complex. 
% In fact, the condition~\eqref{eq:CriterionForLambda} with $\lambda=H$ should be treated in a different way. 
The system of determining equations for conservation-law characteristics of~\eqref{eq:w:tx} 
is more complicated than the above one. 
In summary, from the computational point of view, in order to compute all local conservation laws of~\eqref{eq:w:tx}
it turns out to be optimal to directly work with the conserved currents in the light cone variables, 
to construct the corresponding characteristics using the characteristic form of conservation laws, 
find the subspace of trivial conserved currents as those associated with the zero characteristic of the form~\eqref{eq:w:dal:CL:char} 
and then pull back all the obtained objects by the transformation~\eqref{eq:chvar}.

We considered one more form of the (1+1)-dimensional wave equation, $\mathcal V$: $v_{yz}=v_{yy}$, 
which is related to the canonical form~\eqref{eq:w:dal} in the light-cone variables 
by the point transformation $x=y+z$, $\eta=z$, $w(\xi,\eta)=v(y,z)$.
The equation~$\mathcal V$ has a general Kovalevskaya form when solved for $v_{yy}$. 
It is obvious that the space of its local conservation laws 
is isomorphic to the one for the PDE~\eqref{eq:w:dal} derived in Section~\ref{sec:2}, 
through the above transformation. 
However, if one chooses $v_{yz}$ as the leading derivative in~$\mathcal V$ 
confines differential functions of~$v$ to solutions of~$\mathcal V$ 
by excluding the corresponding principal derivatives $v_{kl}$, $k,l\in\mathbb N$, from the dependence, 
then one can show that a part of conservation-law characteristics is lost 
in the course of looking for the solutions of the corresponding determining equations 
among the confined differential functions. 
More specifically, the subfamily of characteristics corresponding to the ``intact" variable $\xi=y$  and the arbitrary function $G$ survives, 
whereas most of characteristics in the subfamily corresponding to the arbitrary function $F$ in Lemma~\ref{lem:wDal:chars} are lost: 
surviving characteristics from the former subfamily do not contain even-order derivatives 
and can only involve odd derivatives of order three and higher in a linear manner.
We will present details of this consideration in a different paper. 

The example of the equation~$\mathcal V$ illustrates that a conservation-law characteristic can fails to remain a characteristic (and become just a cosymmetry) 
after the substitution of principal derivatives on the solution manifold of the given system of differential equations, 
if the latter is not written in an extended Kovalevskaya form.

We also plan to study similar effects that are related to excluding derivatives 
in view of canonical form potential systems in the course of looking for potential conservation laws 
through their characteristics \cite{Bluman&Cheviakov&Ivanova2006,Kunzinger&Popovych2008}.

%\todo the Klein--Gordon equation???

\medskip
\noindent{\bf Acknowledgements.}
The research of ROP was supported by the Austrian Science Fund (FWF), projects P25064, P29177 and P30233.
AFC is grateful to NSERC for support through a Discovery grant.

\end{document}